\documentclass{birkmult}
\usepackage{amssymb,amsmath}
 \newtheorem{theorem}{Theorem}[section]
 
 \newtheorem{lemma}[theorem]{Lemma}
 
 \theoremstyle{definition}
 
 \theoremstyle{remark}
 \newtheorem{remark}[theorem]{Remark}
 \newtheorem{example}{Example}
 
 \numberwithin{equation}{section}

\begin{document}
\title[Constrained energy problems]
 {Constrained energy problems\\ with external fields}
%----------Author 1
\author[N.~Zorii]{Natalia Zorii}

\address{%
Institute of Mathematics\\
National Academy of Sciences of Ukraine\\
3 Tereshchenkivska Str.\\
01601, Kyiv-4\\
Ukraine}

\email{natalia.zorii@gmail.com}

\subjclass{31C15}

\keywords{Constrained energy problems, external fields, equilibrium
measures, variational inequalities for weighted equilibrium
potentials}

\date{December 31, 2009}

\begin{abstract}
Given a positive definite kernel in a locally compact space~$\mathrm
X$, a closed set~$\Sigma$, a measure $\sigma\geqslant0$, and a
positive continuous~$g$, we study the minimal energy problem in the
presence of an external field~$f$ over the class of all measures
$\nu\geqslant0$ supported by~$\Sigma$ and such that $\int
g\,d\nu=1$, $\sigma-\nu\geqslant0$. Under general assumptions, we
establish the existence of a minimizing
measure~$\lambda^\sigma_\Sigma$ and analyze its continuity
properties in the weak$*$ and strong topologies when $\sigma$
and~$\Sigma$ are varied. We also give a description of the
$f$-weighted potential of~$\lambda^\sigma_\Sigma$ and single out its
characteristic properties. Such results are mostly new even for
classical kernels in $\mathbb R^n$, which is important in
applications.
\end{abstract}

\maketitle

\section{Introduction and statement of the problem}
Let $\mathrm X$ be a locally compact Hausdorff space and $\mathfrak
M=\mathfrak M(\mathrm X)$ the linear space of all real-valued Radon
measures~$\nu$ on~$\mathrm X$ equipped with the {\it vague\/}
\mbox{(${}=${\it weak}$*$)} topology, i.\,e., the topology of
pointwise convergence on the class of all real-valued continuous
functions on~$\mathrm X$ with compact support.

A {\it kernel\/}~$\kappa$ on $\mathrm X$ is meant to be an element
from $\mathrm\Phi(\mathrm X\times\mathrm X)$, where
$\mathrm\Phi(\mathrm Y)$ consists of all lower semicontinuous
functions~$\psi:\mathrm Y\to(-\infty,\infty]$ such that
$\psi\geqslant0$ unless $\mathrm Y$ is compact. Given
$\nu,\,\mu\in\mathfrak M$, the {\it mutual energy\/} and the {\it
potential\/} with respect to the kernel~$\kappa$ are defined by
\[\kappa(\nu,\mu):=\int\kappa(x,y)\,d(\nu\otimes\mu)(x,y)\quad\mbox{and
\ }\kappa_\nu(\cdot):=\int\kappa(\cdot,y)\,d\nu(y),\] respectively.
(Here and in the sequel, when introducing notation, we shall always
tacitly assume the corresponding object on the right to be well
defined.) For $\nu=\mu$ we get the {\it energy\/} $\kappa(\nu,\nu)$
of~$\nu$. Let $\mathcal E$ consist of all $\nu\in\mathfrak M$ with
$-\infty<\kappa(\nu,\nu)<\infty$.

In this work we consider a {\it positive definite\/}
kernel~$\kappa$, which means that it is symmetric (i.\,e.,
$\kappa(x,y)=\kappa(y,x)$ for all $x,\,y\in\mathrm X$) and the
energy $\kappa(\nu,\nu)$, $\nu\in\mathfrak M$, is nonnegative
whenever defined. Then $\mathcal E$ forms a pre-Hil\-bert space with
the scalar product $\kappa(\nu,\mu)$ and the seminorm
$\|\nu\|:=\sqrt{\kappa(\nu,\nu)}$ (see~\cite{F1}). The topology
on~$\mathcal E$ defined by means of this seminorm is called {\it
strong\/}.

For an arbitrary closed set $E\subset\mathrm X$, let $\mathfrak
M^+(E)$ consist of all nonnegative $\nu\in\mathfrak M$ with the
support $S_\nu\subset E$, and let $\mathcal E^+(E):=\mathfrak
M^+(E)\cap\mathcal E$. Given a measure~$\nu$ and a function~$\psi$,
for the sake of brevity we shall write
$\langle\psi,\nu\rangle:=\int\psi\,d\nu$.

Fix an {\it external field\/}~$f$. We assume that either
$f\in\mathrm\Phi(\mathrm X)$ (Case~I), or $f=\kappa_\zeta$, where
$\zeta\in\mathcal E$ is a {\it signed\/} measure (Case~II). Then the
$f$-{\it weighted potential\/}~$W^f_\nu$ and the $f$-{\it weighted
energy\/}~$G_f(\nu)$ of $\nu\in\mathcal E$ are respectively given by
the formulas
\[W^f_\nu(x):=\kappa_\nu(x)+f(x),\qquad G_f(\nu):=\|\nu\|^2+2\langle f,\nu\rangle=
\langle W^f_\nu+f,\nu\rangle.\] Note that $W^f_\nu$, $\nu\in\mathcal
E$, is defined and ${}\ne-\infty$ at least {\it nearly everywhere\/}
(n.\,e.) in~$\mathrm X$, that is, except at most for some set
$N\subset\mathrm X$ with the interior capacity $C(N)=0$.

Having fixed also a nonempty closed set $\Sigma\subset\mathrm X$, we
consider a function~$g>0$, defined and continuous at least in some
open neighborhood~$U_\Sigma$ of~$\Sigma$, and a measure
$\sigma\in\mathfrak M^+(\Sigma)$ which will serve as a {\it
constraint\/}.

We are interested in the {\it constrained minimal $f$-weighted
energy problem\/}
\begin{equation}\label{gauss}G_f^\sigma(\Sigma,g):=\inf_{\nu\in\mathcal
E^\sigma(\Sigma,g)}\,G_f(\nu),\end{equation} where \[\mathcal
E^\sigma(\Sigma,g):=\bigl\{\nu\in\mathcal E^+(\Sigma):\ \langle
g,\nu\rangle=1, \ \nu\leqslant\sigma\bigr\}\] and
$\nu\leqslant\sigma$ means that $\sigma-\nu\geqslant0$.
(In~(\ref{gauss}), as usual, the infimum over the empty set is taken
to be~$+\infty$.) Along with its electrostatic interpretation, this
problem has also found applications in approximation
theory~(see~\cite{D,DS,R}). If
\begin{equation}\label{finite}G_f^\sigma(\Sigma,g)<\infty\end{equation}
(or, which is equivalent, if the class
\begin{equation}\label{gf}\mathcal
E_f^\sigma(\Sigma,g):=\bigl\{\nu\in\mathcal E^\sigma(\Sigma,g):\
G_f(\nu)<\infty\bigr\}\end{equation} is nonempty), then we shall
consider the problem on the existence of
$\lambda_\Sigma^\sigma\in\mathcal E^\sigma(\Sigma,g)$ with minimal
$f$-weighted energy
$G_f(\lambda_\Sigma^\sigma)=G_f^\sigma(\Sigma,g)$. Such a
$\lambda_\Sigma^\sigma$ (if exists) will be called an {\it
equilibrium measure\/} corresponding to the data~$\kappa$, $\Sigma$,
$\sigma$, $g$, and~$f$.

If $\mathrm X=\mathbb R^2$, $\kappa(x,y)=-\log|x-y|$, and $g=1$, the
constrained energy problem has been analyzed by P.~Dragnev, E.~Saff
and E.~Rakhmanov: see~\cite{D,DS}, where $f\in\mathrm\Phi(\mathbb
R^2)$ is fast growing at infinity, and~\cite{R}, where
$\Sigma=[-1,1]$ and $f=0$.

However, the methods applied in this note and the results obtained
differ essentially from those in~\cite{D,DS,R}. Namely, our approach
is mainly based on the use of both the strong and vague topologies,
which enables us in both Cases~I and~II to establish the existence
of an equilibrium measure~$\lambda_\Sigma^\sigma$ for
noncompact~$\Sigma$ and to study continuity properties
of~$\lambda_\Sigma^\sigma$ as a function of~$(\Sigma,\sigma)$. We
also obtain variational inequalities for the $f$-weighted
equilibrium potential $W_\lambda^f$ and single out its
characteristic properties, modifying properly the arguments
from~\cite{D,DS,R}.

For the sake of simplicity we shall restrict ourselves to the case
where either $\mathrm X$ is a countable union of compact sets or
$\inf_{x\in\mathrm X}\,g(x)>0$. Then the concept of local
$\nu$-negligibility and that of $\nu$-negligibility coincide for any
$\nu\geqslant0$ with $\langle g,\nu\rangle<\infty$; hence, every $N$
with $C(N)=0$ is $\nu$-negligible if, moreover, $\nu\in\mathcal E$.

Before formulating the results obtained, we observe the following
lemma.

\begin{lemma}\label{minuss} $G_f^\sigma(\Sigma,g)>-\infty$.\end{lemma}

\begin{proof} Indeed, in Case~II it is an immediate consequence
of the representation
\begin{equation}\label{repres}
G_f(\nu)=\|\nu\|^2+2\kappa(\nu,\zeta)=\|\nu+\zeta\|^2-\|\zeta\|^2,
\quad\nu\in\mathcal E.\end{equation} Let Case~I take place. If
$\mathrm X$ is compact, then $f\in\mathrm\Phi(\mathrm X)$ is bounded
from below by~$-c$, where $c>0$, while $\nu(\mathrm
X)\leqslant\bigl[\min_{x\in\mathrm X}\,g(x)\bigr]^{-1}<\infty$ for
all $\nu\in\mathcal E^\sigma(\Sigma,g)$, and the required inequality
follows. Otherwise, $f$ has to be ${}\geqslant0$; hence,
$G_f^\sigma(\Sigma,g)\geqslant0$.\end{proof}

\section{Main results}

Following~\cite{F1}, we call a (positive definite) kernel~$\kappa$
{\it perfect\/} if $\mathcal E^+:=\mathcal E^+(\mathrm X)$, treated
as a topological subspace of~$\mathcal E$, is strongly complete and
the strong topology on~$\mathcal E^+$ is finer than the induced
vague topology. It follows that a perfect kernel has to be {\it
strictly positive definite\/}, and the seminorm~$\|\cdot\|$ is then
actually a norm.

\begin{remark} It is well known  (see,
e.\,g.,~\cite{E1,F1,L}) that the class of perfect kernels includes
the Riesz kernels $|x-y|^{\alpha-n}$, $0<\alpha<n$, in~$\mathbb
R^n$, $n\geqslant2$ (in particular, the Newtonian kernel
$|x-y|^{2-n}$ in~$\mathbb R^n$, $n\geqslant3$), the restriction of
the logarithmic kernel $-\log\,|x-y|$ in~$\mathbb R^2$ to the open
unit disk, and the Green kernel~$g_D$, where $D$ is an open set
in~$\mathbb R^n$, $n\geqslant 2$, and~$g_D$ is its generalized Green
function.\end{remark}

Let $\nu_E$ denote the trace of $\nu\in\mathfrak M$ upon a
$\nu$-measurable set~$E$.

\begin{theorem}\label{exist} Assume {\rm(\ref{finite})} to hold, $\kappa$ to be perfect,
and let $C(\Sigma)$ be finite\,\footnote{Even for the Newtonian
kernel, sets of finite capacity might be noncompact
(see~\cite{L}).}. If, moreover, either $g|_{U_\Sigma}$ is bounded or
there exist $r\in(1,\infty)$ and $\omega\in\mathcal E$ such that
\begin{equation}
g^{r}(x)\leqslant\kappa_\omega(x)\quad\mbox{n.\,e. in \ } U_\Sigma,
\label{growth}
\end{equation}
then the following assertions hold true:
\begin{itemize}
\item[\rm(a)] There exists a unique equilibrium measure
$\lambda_\Sigma^\sigma$.\smallskip
\item[\rm(b)] Let $\Sigma_s\subset U_\Sigma$, $s\in S$, be a decreasing ordered family of
closed sets such that $C(\Sigma_s)<\infty$ and $\bigcap_{s\in
S}\,\Sigma_s=\Sigma$. Let $\sigma_s\in\mathfrak M^+(\Sigma_s)$,
$s\in S$, decrease and converge vaguely to~$\sigma$. Then
\begin{equation}\label{sigma}
G^\sigma_{f}(\Sigma,g)=\lim_{s\in
S}\,G^{\sigma_s}_{f}(\Sigma_s,g)\end{equation} and
$\lambda_{\Sigma_s}^{\sigma_s}\to\lambda_{\Sigma}^{\sigma}$ strongly
(hence, also vaguely).\smallskip
\item[\rm(c)] Let $\{K\}$ be the increasing ordered family of all compact subsets of~$\Sigma$.
Then there exists a net $(\beta^*_{K})_{K\in
\{K\}}\subset(1,\infty)$ that decreases to~$1$ and such that, for
all $\beta_{K}\in[1,\beta^*_{K}]$,
\begin{equation}
G^\sigma_{f}(\Sigma,g)=\lim_{K\uparrow\Sigma}\,
G^{\beta_K\sigma_K}_f(K,g).\label{contnew}
\end{equation}
Furthermore, $\lambda^{\beta_K\sigma_K}_K\to\lambda^{\sigma}_\Sigma$
strongly (and, hence, vaguely).
\end{itemize}
\end{theorem}

Given a closed set $E\subset\mathrm X$ with $C(E)>0$ and a
universally measurable function~$\psi$ bounded from below nearly
everywhere in~$E$, we write
$$
"\!\inf_{x\in E}\!"\,\,\psi(x):=\sup\,\bigl\{q: \ \psi(x)\geqslant
q\quad\mbox{n.\,e.~in \ } E\bigr\}.
$$
Then \[\psi(x)\geqslant"\!\inf_{x\in
E}\!"\,\,\psi(x)\quad\mbox{n.\,e.~in \ }E,\] which follows from the
countable subadditivity of $C(\cdot)$ over universally measurable
sets with interior capacity zero~\cite{F1}. If $\psi$ is bounded
from above n.\,e.~in~$E$, write
\["\!\sup_{x\in E}\!"\,\,\psi(x):="\!\inf_{x\in E}\!"\,\,-\psi(x).\]

In the next theorem we assume that $\sigma_K\in\mathcal E$ for every
compact $K\subset\Sigma$ and $\langle g,\sigma_{\Sigma_0}\rangle>1$,
where $\Sigma_0:=\bigl\{x\in\Sigma:\ f(x)<\infty\bigr\}$. Then
(\ref{finite}) necessarily holds, since one can choose a compact
$K_0\subset\Sigma_0$ so that $\sigma_{K_0}\bigl/\langle
g,\sigma_{K_0}\rangle\in\mathcal E_f^\sigma(\Sigma,g)$;
cf.~\cite{DS,Z5}.

\begin{theorem}\label{varin}
Given $\lambda\in\mathcal E^\sigma(\Sigma,g)$, the following
assertions are equivalent:
\begin{itemize}
\item[\rm(i)] $\lambda$ is an
equilibrium measure $\lambda^{\sigma}_\Sigma$.\smallskip
\item[\rm(ii)] There exists $w_\lambda\in\mathbb R$ such
that\,\footnote{Observe that, under the assumptions made, both
$C\bigl(S_{\sigma-\lambda}\bigr)$ and $C\bigl(S_\lambda\bigr)$ are
nonzero.}
\begin{align}\label{ineq1}W^f_\lambda(x)&\geqslant w_\lambda\,
g(x)\quad\mbox{n.\,e.~in \ }S_{\sigma-\lambda},\\
\label{ineq2}W^f_\lambda(x)&\leqslant w_\lambda\,
g(x)\quad\mbox{n.\,e.~in \ }S_{\lambda}.\end{align}
\item[\rm(iii)] $-\infty<\ell\leqslant L<\infty$, where
\begin{equation}\label{lL}\ell:="\!\!\sup_{x\in
S_{\lambda}}\!\!"\,\,\,\frac{W^f_\lambda(x)}{g(x)}\,,\qquad
L:=\,\,"\!\!\!\!\!\inf_{x\in
S_{\sigma-\lambda}}\!\!\!\!\!"\,\,\,\frac{W^f_\lambda(x)}{g(x)}\,.\end{equation}
\end{itemize}
\end{theorem}

\begin{remark} It follows that, if $\lambda$ is an equilibrium measure, then the collection of all $w_\lambda$
for whom both (\ref{ineq1}) and~(\ref{ineq2}) hold forms the finite
closed interval~$[\ell,L]$. Of course, if $g=1$, $f=0$ and $\kappa$
satisfies the maximum principle, then $\ell=L$ and
$[\ell,L]$~consists of just one point. However, this is not the case
in general (see~Sec.~\ref{examples}).
\end{remark}

\begin{remark}Relation (\ref{ineq2}) actually holds for every $x\in
S_\lambda$ if, moreover, $f\in\mathrm\Phi(\mathrm X)$.\end{remark}

The rest of the article is organized as follows. Theorem~\ref{varin}
will be proved in~Sec.~\ref{proof:varin}. The proof of
Theorem~\ref{exist}, to be given in Sec.~\ref{prooff}, is based on a
theorem on the strong completeness of $\mathcal E^\sigma(\Sigma,g)$,
which is the main subject of the next section.

\section{Auxiliary assertions}\label{auxiliary}

\begin{theorem}\label{lemma:exist'} Let $\kappa$ be perfect,
$E\subset U_\Sigma$ be a closed set with $C(E)<\infty$, and let
$g$~be as in~Theorem~{\rm\ref{exist}}. Then
\[\mathcal E^+(E,g):=\bigl\{\nu\in\mathcal E^+(E):\ \langle
g,\nu\rangle=1\bigr\},\] treated as a topological subspace
of~$\mathcal E$, is strongly complete. In more detail, every
strongly fundamental net $(\nu_s)_{s\in S}\subset\mathcal E^+(E,g)$
converges strongly (and, hence, vaguely) to a unique
$\nu_0\in\mathcal E^+(E,g)$. If, moreover, $\sigma_0\in\mathfrak
M^+(E)$ is given, then the same holds true for $\mathcal
E^{\sigma_0}(E,g)$ instead of $\mathcal E^+(E,g)$.
\end{theorem}

\begin{proof} For every $B\subset E$ there exists a uniquely determined measure
$\theta_B\in\mathcal E^+(\,\overline{B}\,)$, called the interior
capacitary distribution associated with~$B$, with the properties
\begin{equation}
\theta_B(\mathrm X)=\|\theta_B\|^2=C(B), \label{5}
\end{equation}
\begin{equation}
\kappa_{\theta_B}(x)\geqslant1\quad\mbox{n.\,e. in \ } B. \label{6}
\end{equation}
Indeed, this follows from $C(E)<\infty$ and the perfectness
of~$\kappa$  due to~\cite[Th.~4.1]{F1}.

One can certainly assume that $C(E)>0$, since otherwise $\mathcal
E^+(E,g)$ is empty. Also observe that there is no loss of generality
in assuming~$g$ to satisfy~(\ref{growth}) for~$E$ instead
of~$U_\Sigma$, since otherwise $g|_E$ is bounded from above (say
by~$M$), which combined with~(\ref{6}) again gives~(\ref{growth})
for $\omega:=M^{r}\,\theta_{E}$, $r\in(1,\infty)$ being arbitrary.

Fix a strongly fundamental net $(\nu_s)_{s\in S}\subset\mathcal
E^+(E,g)$; then one can assume it to be strongly bounded. Due to the
perfectness of the kernel, such a net converges to some
$\nu_0\in\mathcal E^+$ strongly and, therefore, vaguely. The latter
yields $S_{\nu_0}\subset E$ and $\langle g,\nu_0\rangle\leqslant1$.
To prove that $\mathcal E^+(E,g)$ is strongly complete, it is enough
to show that
\begin{equation}\label{24}\langle
g,\nu_0\rangle=1.\end{equation}

To this end, we shall treat $E$ as a locally compact space with the
topology induced from~$\mathrm X$. Given a set $B\subset E$, let
$\chi_B$ denote its characteristic function and let $CB:=E\setminus
B$. Further, let $\{K\}$ be the increasing family of all compact
subsets~$K$ of~$E$. Since $g\chi_{K}$ is upper semicontinuous on~$E$
while $(\nu_s)_{s\in S}$ converges to~$\nu_0$ vaguely, for every
$K\in\{K\}$ we have
\[
\langle g\chi_{K},\nu_0\rangle\geqslant\limsup_{s\in S}\,\langle
g\chi_{K},\nu_s\rangle.\] On the other hand, Lemma~1.2.2
from~\cite{F1} gives
\[
\langle g,\nu_0\rangle=\lim_{K\in\{K\}}\,\langle
g\chi_{K},\nu_0\rangle.\] Combining the last two relations, we
obtain
\[
1\geqslant\langle g,\nu_0\rangle\geqslant\limsup_{(s,\,K)\in
S\times\{K\}}\,\langle g\chi_{K},\nu_s\rangle= 1-\liminf_{(s,\,K)\in
S\times\{K\}}\,\langle g\chi_{CK},\nu_s\rangle,\] $S\times\{K\}$
being the directed product of the directed sets~$S$ and~$\{K\}$.
Hence, if we prove
\begin{equation}
\liminf_{(s,\,K)\in S\times\{K\}}\,\langle
g\chi_{CK},\nu_s\rangle=0, \label{25}
\end{equation}
the desired relation (\ref{24}) follows.

To obtain (\ref{25}), consider the interior capacitary distribution
$\theta_{CK}$, $K\in\{K\}$ being given. Then application
of~Lemma~4.1.1 and Theorem~4.1 from~\cite{F1} yields
\[
\|\theta_{CK}-\theta_{C\tilde{K}}\|^2\leqslant
\|\theta_{CK}\|^2-\|\theta_{C\tilde{K}}\|^2\quad\mbox{provided \
}K\subset\tilde{K}.\] Furthermore, it is clear from~(\ref{5}) that
the net $\|\theta_{CK}\|$, $K\in\{K\}$, is bounded and
nonincreasing, and hence fundamental in~$\mathbb R$. The preceding
inequality thus implies that $(\theta_{CK})_{K\in\{K\}}$ is strongly
fundamental in~$\mathcal E^+$. Since it converges vaguely to zero,
zero is also its strong limit due to the perfectness of the kernel;
hence,
\begin{equation*}
\lim_{K\in\{K\}}\,\|\theta_{CK}\|=0. \label{27}
\end{equation*}

Write $q:=r(r-1)^{-1}$, where $r\in(1,\infty)$ is the number
involved in condition~(\ref{growth}). Combining (\ref{growth}) with
(\ref{6}) shows that the inequality
\[
g(x)\,\chi_{CK}(x)\leqslant\kappa_\omega(x)^{1/r}\,
\kappa_{\theta_{CK}}(x)^{1/q}\]  subsists n.\,e.~in~$E$, and hence
$\nu_s$-almost everywhere in~$\mathrm X$. Having integrated this
relation with respect to~$\nu_s$, we then apply the H\"older and,
subsequently, the Cauchy-Schwarz inequalities to the integrals on
the right. This gives \[\langle
g\chi_{CK},\nu_s\rangle\leqslant\langle\kappa_\omega,\nu_s\rangle^{1/r}\,
\langle\kappa_{\theta_{CK}},\nu_s\rangle^{1/q}\leqslant
\|\omega\|^{1/r}\,\|\theta_{CK}\|^{1/q}\,\|\nu_s\|.\] Taking limits
here along $S\times\{K\}$, we obtain~(\ref{25}) and,
hence,~(\ref{24}).

It has thus been proved that $\mathcal E^+(E,g)$ is strongly
complete. To establish the strong completeness of $\mathcal
E^{\sigma_0}(E,g)$, it is therefore enough to note that the set of
all $\nu\in\mathcal E^+(E)$ which do not exceed~$\sigma_0$ is
vaguely (hence, also strongly) closed.\end{proof}

\begin{lemma}\label{stronglower}Assume $\kappa$ to be perfect. In both Cases I and II,
the $f$-weighted energy~$G_f$ is lower semicontinuous on~$\mathcal
E^+$ in the strong topology.\end{lemma}

\begin{proof}Actually, in Case~II \,$G_f$ is continuous on~$\mathcal E^+$ in the strong topology,
which is seen from~(\ref{repres}). Let Case~I take place; then
$f\in\mathrm\Phi(\mathrm X)$ and, hence, $\langle f,\nu\rangle$ is
vaguely lower semicontinuous on~$\mathcal E^+$ (see~\cite{F1}).
Since so is $\kappa(\nu,\nu)$, the desired conclusion follows in
view of the fact that the strong topology is finer than the vague
one.\end{proof}

\begin{lemma}\label{lequiv} Assume that {\rm(\ref{finite})} holds. For
$\lambda\in\mathcal E_f^\sigma(\Sigma,g)$ to be an equilibrium
measure, it is necessary and sufficient that
\begin{equation}\label{lchar}\langle
W_f^\lambda,\nu-\lambda\rangle\geqslant0\quad\mbox{for all \ }
\nu\in\mathcal E_f^\sigma(\Sigma,g).\end{equation}\end{lemma}

\begin{proof}Since $\mathcal E_f^\sigma(\Sigma,g)$ is convex, for any its elements $\nu,\,\mu$ and
$h\in(0,1]$ we get
\begin{equation}\label{mainin}G_f\bigl(h\nu+(1-h)\mu\bigr)-G_f(\mu)=2h\langle
W_f^\mu,\nu-\mu\rangle+h^2\|\nu-\mu\|^2.\end{equation} (It has been
used here that $G_f$ is finite on $\mathcal E_f^\sigma(\Sigma,g)$;
see~(\ref{gf}) and Lemma~\ref{minuss}.) If $\mu=\lambda$ is an
equilibrium measure, then the left (hence, the right) side
of~(\ref{mainin}) is~${}\geqslant0$, which leads to~(\ref{lchar}) by
letting $h\to0$. Conversely, if (\ref{lchar}) holds, then
(\ref{mainin}) with $\mu=\lambda$ and $h=1$ gives $G_f(\nu)\geqslant
G_f(\lambda)$ for all $\nu\in\mathcal E_f^\sigma(\Sigma,g)$, as
required.\end{proof}

\section{Proof of Theorem~\ref{exist}}\label{prooff}

{(a)} \ Fix $(\nu_s)_{s\in S}\subset\mathcal E_f^\sigma(\Sigma,g)$
with the property that $\lim_{s\in
S}\,G_f(\nu_s)=G^\sigma_f(\Sigma,g)$; such a net will be called {\it
minimizing\/}. Then identity (\ref{mainin}) with $h=1/2$ implies
\[\|\nu_s-\nu_d\|^2\leqslant2G_f(\nu_s)+2G_f(\nu_d)-4G^\sigma_f(\Sigma,g)\quad\mbox{for
all \ } s,\,d\in S,\] which establishes the strong fundamentality of
$(\nu_s)_{s\in S}$ when combined with the above definition and
Lemma~\ref{minuss}. Therefore, by Theorem~\ref{lemma:exist'}, it
converges strongly and vaguely to a unique $\nu_0\in\mathcal
E^\sigma(\Sigma,g)$. On account of Lemma~\ref{stronglower}, we thus
get
\begin{equation*}\label{proofi}G^\sigma_f(\Sigma,g)\leqslant G_f(\nu_0)\leqslant\liminf_{s\in
S}\,G_f(\nu_s)=G^\sigma_f(\Sigma,g);\end{equation*}  consequently,
$\nu_0$ is an equilibrium measure~$\lambda^\sigma_\Sigma$.

The uniqueness of~$\lambda^\sigma_\Sigma$ follows in a standard way.
Indeed, if $\lambda,\,\hat\lambda\in\mathcal E_f^\sigma(\Sigma,g)$
are two equilibrium measures, then the sequence
$(\mu_n)_{n\in\mathbb N}$ with $\mu_{2n}=\lambda$ and
$\mu_{2n+1}=\hat\lambda$ is minimizing; therefore, what has just
been proved yields $\lambda=\hat\lambda$ as required.\medskip

\noindent{(b)} \ Under the assumptions of~(b), $\mathcal
E^\sigma(\Sigma,g)\subset\mathcal
E^{\sigma_d}(\Sigma_d,g)\subset\mathcal E^{\sigma_s}(\Sigma_s,g)$
for all $s,\,d\in S$ whenever $s\leqslant d$. Hence,
$G^{\sigma_s}_{f}(\Sigma_s,g)$ increases as $s$ ranges through~$S$
and
\begin{equation}\label{Gfund}G^{\sigma}_f(\Sigma,g)\geqslant
\lim_{s\in S}\,G^{\sigma_s}_f(\Sigma_s,g).\end{equation} By reason
of~(\ref{finite}), this yields $G^{\sigma_s}_f(\Sigma_s,g)<\infty$
for every $s\in S$. Therefore, by~(a), there exists a unique
equilibrium measure $\lambda_s:=\lambda^{\sigma_s}_{\Sigma_s}$.
Since $\lambda_d\in\mathcal E_f^{\sigma_s}(\Sigma_s,g)$ for all
$d\geqslant s$, we conclude from Lemma~\ref{lequiv} that $\langle
W_f^{\lambda_s},\lambda_d- \lambda_s\rangle\geqslant0$ and,
consequently,
\begin{equation}\label{convex}
\|\lambda_d-\lambda_s\|^2\leqslant G^{\sigma_d}_f(\Sigma_d,g)-
G^{\sigma_s}_f(\Sigma_s,g).
\end{equation}
However, as follows from (\ref{Gfund}), the net
$G^{\sigma_s}_f(\Sigma_s,g)$, $s\in S$, is fundamental in~$\mathbb
R$. When combined with~(\ref{convex}), this implies that
$(\lambda_s)_{s\geqslant\ell}$ is strongly fundamental in~$\mathcal
E^{\sigma_\ell}(\Sigma_\ell,g)$ for every $\ell\in S$. Therefore, by
Theorem~\ref{lemma:exist'}, $(\lambda_s)_{s\in S}$ converges
strongly and vaguely to a unique measure $\nu_0$ and
$\nu_0\in\mathcal E^{\sigma_\ell}(\Sigma_\ell,g)$ for every $\ell\in
S$. Since $\nu_0\leqslant\sigma_\ell$ and
$\sigma_\ell-\nu_0\to\sigma-\nu_0$ vaguely as $\ell$ ranges
over~$S$, we get $\nu_0\leqslant\sigma$. Thus, actually
$\nu_0\in\mathcal E^\sigma(\Sigma,g)$ and, by
Lemma~\ref{stronglower},
\[G_f(\Sigma,g)\leqslant G_f(\nu_0)\leqslant\liminf_{s\in S}\,G_f(\lambda_s).\]
Together with~(\ref{Gfund}), this gives (\ref{sigma}) and
$\nu_0=\lambda^{\sigma}_{\Sigma}$, and the proof of~(b) is
complete.\medskip

\noindent{(c)} \ To prove (c), we start by establishing the relation
\begin{equation}
G^\sigma_f(\Sigma,g)=\lim_{K\uparrow\Sigma}\,
G^{\sigma_K}_f(K,g).\label{cont}
\end{equation}
For every $\nu\in\mathcal E_f^\sigma(\Sigma,g)$, write
$\hat{\nu}_K:=\nu_K/\langle g,\nu_K\rangle$. Since, by
\cite[Lemma~1.2.2]{F1},
\[1=\lim_{K\uparrow\Sigma}\,\langle
g,\nu_K\rangle,\quad\langle
f,\nu\rangle=\lim_{K\uparrow\Sigma}\,\langle f,\nu_K\rangle,\quad
\|\nu\|^2=\lim_{K\uparrow\Sigma}\,\|\nu_K\|^2,\] we obtain
\begin{equation}
G_f(\nu)=\lim_{K\uparrow\Sigma}\,G_f(\hat{\nu}_K).\label{4w}\end{equation}
Having fixed $\varepsilon>0$, we also conclude that there exists
$K^0\in\{K\}$ such that $\hat{\nu}_K\in\mathcal
E_f^{(1+\varepsilon)\sigma_K}(K,g)$ for all $K\in\{K\}$ that
follow~$K^0$. This yields
\begin{equation}
G_f(\hat{\nu}_K)\geqslant
G^{(1+\varepsilon)\sigma_K}_f(K,g).\label{www}\end{equation} In view
of the arbitrary choice of~$\nu$, substituting (\ref{www})
into~(\ref{4w}) gives
\[
G^\sigma_f(\Sigma,g)\geqslant
\lim_{K\uparrow\Sigma}\,G^{(1+\varepsilon)\sigma_K}_f(K,g)\geqslant
G^{(1+\varepsilon)\sigma}_f(\Sigma,g),
\]
the latter inequality being a consequence of the monotonicity of
$G^\sigma_f(\cdot,g)$. Letting here $\varepsilon\to0$ and
applying~(b), we obtain
\[G^\sigma_f(\Sigma,g)=\lim_{\varepsilon\to0}\,
\Bigl[\,\lim_{K\uparrow\Sigma}\,
G^{(1+\varepsilon)\sigma_K}_f(K,g)\,\Bigr]=
\lim_{K\uparrow\Sigma}\,G^{\sigma_K}_f(K,g),\] and (\ref{cont}) is
thus proved. Since obviously $\lambda^{\sigma_K}_K\in\mathcal
E_f^\sigma(\Sigma,g)$, relation~(\ref{cont}), in turn, implies that
the net $(\lambda^{\sigma_K}_K)_{K\in\{K\}}$ is minimizing and,
hence, strongly fundamental.

Further, according to (b), for every $K\in\{K\}$ one can choose
$\beta^*_K\in(1,\infty)$ so that $\beta^*_K\downarrow1$ as
$K\uparrow\Sigma$ and, for all $\beta_K\in[1,\beta^*_K]$,
\begin{equation}\label{last1}\lim_{K\in\{K\}}\,
\|\lambda^{\beta_K\sigma_K}_K-\lambda^{\sigma_K}_K\|^2=0,\end{equation}
\begin{equation}\label{last2}\lim_{K\in\{K\}}\,
\bigl[G_f(\lambda^{\beta_K\sigma_K}_K)-
G_f(\lambda^{\sigma_K}_K)\bigr]=0.\end{equation} Then combining
(\ref{cont}) and~(\ref{last2}) gives~(\ref{contnew}), while
(\ref{last1}) together with the strong fundamentality of
$(\lambda^{\sigma_K}_K)_{K\in\{K\}}$ shows that
$(\lambda^{\beta_K\sigma_K}_K)_{K\in\{K\}}$ is strongly fundamental
as well. Hence, according to Theorem~\ref{lemma:exist'}, there
exists a unique~$\nu_0$ which is the strong limit of
$(\lambda^{\beta_K\sigma_K}_K)_{K\in\{K\}}$ and belongs to~$\mathcal
E^{(1+\delta)\sigma}(\Sigma,g)$ for every $\delta>0$; therefore,
$\nu_0\in\mathcal E^\sigma(\Sigma,g)$. On account of
Lemma~\ref{stronglower} and~(\ref{contnew}), this yields
\[G_f^\sigma(\Sigma,g)\leqslant G_f(\nu_0)\leqslant\lim_{s\in
S}\,G_f(\lambda^{\beta_K\sigma_K}_K)=G_f^\sigma(\Sigma,g).\]
Consequently, $\nu_0=\lambda^\sigma_\Sigma$, and the proof is
complete.\qed

\section{Proof of Theorem~\ref{varin}}\label{proof:varin}

Assume (i) to hold. Since $G_f(\lambda)$ is finite, so is $\langle
W^f_\lambda,\lambda\rangle$. We start by showing that
\begin{equation}\label{wfl}W^f_\lambda(x)\geqslant\langle
W^f_\lambda,\lambda\rangle\,g(x)\quad\mbox{n.\,e.~in \
}S_{\sigma-\lambda}.\end{equation} On the contrary, let $C(N)>0$,
where $N:=\bigl\{x\in S_{\sigma-\lambda}:\ W^f_\lambda(x)<\langle
W^f_\lambda,\lambda\rangle\,g(x)\bigr\}$. It follows from
\cite[Th.~4.2]{F1} that then one can choose $n\in\mathbb N$ and a
compact set $K\subset N$ with $C(K)>0$ so that
$W^f_\lambda(x)\bigl/g(x)\leqslant\langle
W^f_\lambda,\lambda\rangle-n^{-1}$ for all $x\in K$. Write
$\tau:=\beta(\sigma-\lambda)_K$, where $\beta:=1\bigl/\langle
g,(\sigma-\lambda)_K\rangle$. Then $\tau$ belongs to~$\mathcal E^+$,
is ${}\ne0$, and
\begin{equation}\label{add1}\langle W^f_\lambda,\tau\rangle<\langle
W^f_\lambda,\lambda\rangle.\end{equation} Since
$\kappa(\lambda,\tau)$ is finite, this yields $\langle
f,\tau\rangle<\infty$. A staightforward verification also shows that
$\tau_h:=(1-h)\lambda+h\tau\leqslant\sigma$ for any $h\in(0,1]$.
Consequently, $\tau_h\in\mathcal E_f^\sigma(\Sigma,g)$ and,
by~Lemma~\ref{lequiv}, $\langle
W^f_\lambda,\tau_h-\lambda\rangle=h\langle
W^f_\lambda,\tau-\lambda\rangle\geqslant0$, which
contradicts~(\ref{add1}).

Thus, according to~(\ref{wfl}), $W^f_\lambda/g(x)$ is bounded from
below n.\,e.~in~$S_{\sigma-\lambda}$; this implies~(\ref{ineq1})
with $w_\lambda=L$, where $L$ is defined by~(\ref{lL}). In turn,
(\ref{ineq1}) yields $L<\infty$, because
$C\bigl(S_{\sigma-\lambda}\cap\Sigma_0\bigr)>0$. Hence,
$\infty>L\geqslant\langle W^f_\lambda,\lambda\rangle>-\infty$.

We proceed by establishing (\ref{ineq2}) with $w_\lambda=L$. Having
denoted (cf.~\cite{R})
\[E^+(w):=\bigl\{x\in\Sigma:\ W^f_\lambda(x)\bigl/g(x)>w\bigr\},\quad
E^-(w):=\bigl\{x\in\Sigma:\ W^f_\lambda(x)\bigl/g(x)<w\bigr\},\]
where $w\in\mathbb R$ is arbitrary, we assume on the contrary that
(\ref{ineq2}) for $w_\lambda=L$ does not hold. Then
$\lambda\bigl(E^+(L)\bigr)>0$; hence,
$\lambda\bigl(E^+(w_1)\bigr)>0$ for some $w_1\in(L,\infty)$.

At the same time, as $w_1>L$, relation (\ref{ineq1}) yields
$(\sigma-\lambda)\bigl(E^-(w_1)\bigr)>0$. Therefore, there is a
compact set $F\subset E^-(w_1)$ such that $\xi:=(\sigma-\lambda)_F$
is nonzero. Since $\xi\in\mathcal E^+$ and $\langle
W^f_\lambda,\xi\rangle\leqslant w_1\langle g,\xi\rangle<\infty$, we
get $\langle f,\xi\rangle<\infty$. A direct verification also shows
that
\[\gamma:=\lambda-\lambda_{E^+(w_1)}+\alpha\xi\leqslant\sigma,\quad\mbox{where
\ }\alpha:=\langle g,\lambda_{E^+(w_1)}\rangle\bigl/\langle
g,\xi\rangle.\] Consequently, $\gamma\in\mathcal
E_f^\sigma(\Sigma,g)$. On the other hand, it also follows from the
above that
\begin{equation*}
\langle W_f^\lambda,\gamma-\lambda\rangle=\langle
W_f^\lambda-w_1g,\gamma-\lambda\rangle=-\langle
W_f^\lambda-w_1g,\lambda_{E^+(w_1)}\rangle+\alpha\langle
W_f^\lambda-w_1g,\xi\rangle<0,\end{equation*} which is, however,
impossible (see~Lemma~\ref{lequiv}). Thus,
$\mbox{(i)}\Rightarrow\mbox{(ii)}$.

Furthermore, since $L$ is finite, (\ref{ineq2}) with $w_\lambda=L$
yields $\ell\leqslant L$. To complete the proof of~(iii), it remains
to observe that $\ell>-\infty$, which is obtained from~(\ref{ineq2})
with $w_\lambda=\ell$ due to the fact that $W^f_\lambda\ne-\infty$
n.\,e.~in~$\Sigma$. Hence, $\mbox{(i)}\Rightarrow\mbox{(iii)}$.

Next, assume $\lambda\in\mathcal E^\sigma(\Sigma,g)$ to
satisfy~(\ref{ineq1}) and~(\ref{ineq2}) for some
$w_\lambda\in\mathbb R$. Then actually $\lambda\in\mathcal
E_f^\sigma(\Sigma,g)$, which is seen from~(\ref{ineq2}) when
integrated with respect to~$\lambda$. Given $\nu\in\mathcal
E_f^\sigma(\Sigma,g)$, we also conclude from~(\ref{ineq1})
and~(\ref{ineq2}) that
\begin{eqnarray*}\bigl\langle
W_f^\lambda,\nu-\lambda\bigr\rangle&\!\!\!\!\!=\!\!\!\!\!&\bigl\langle
W_f^\lambda-w_\lambda\,g,\nu-\lambda\bigr\rangle\\
&\!\!\!\!\!{}=\!\!\!\!\!&\bigl\langle W_f^\lambda-w_\lambda\,
g,\nu_{E^+(w_\lambda)}\bigr\rangle+\bigl\langle
W_f^\lambda-w_\lambda\,
g,(\nu-\sigma)_{E^-(w_\lambda)}\bigr\rangle\geqslant0,\end{eqnarray*}
which establishes (i) according to~Lemma~\ref{lequiv}. Thus,
$\mbox{(ii)}\Rightarrow\mbox{(i)}$.

Since (iii) obviously yields (ii) for any $w_\lambda\in[\ell,L]$,
the proof is complete.\qed

\section{Examples}\label{examples}
The following easily verified fact is used in this section: if
$\lambda_*$ gives a solution to the {\it unconstrained\/}
$f$-weighted minimal energy problem over a closed set $\Sigma_*$,
i.\,e.,
\[\lambda_*\in\mathcal E^+(\Sigma_*,g),\qquad
G_f(\lambda_*)=\min_{\nu\in\mathcal E^+(\Sigma_*,g)}\,G_f(\nu),\]
then $\lambda_*$ also serves as an equilibrium
measure~$\lambda^\sigma_\Sigma$, provided the closed set $\Sigma$
and the constraint $\sigma\in\mathfrak M^+(\Sigma)$ satisfy the
assumptions $S_{\lambda_*}\subset\Sigma\subset\Sigma_*$ and
$\sigma\geqslant\lambda_*$.

In the examples below, the collection of all $w_\lambda$ for whom
both (\ref{ineq1}) and~(\ref{ineq2}) hold forms the whole
non-degenerated interval $[\ell,L]$. The kernels from the examples
are perfect, so that every equilibrium measure is determined
uniquely.

\begin{example}\label{1}Let $\mathrm X=\mathbb R^n$, $n\geqslant3$, $g=1$, $f=0$,
$\kappa(x,y)=|x-y|^{\alpha-n}$, where $\alpha\in(2,n)$ is given, and
let $\Sigma:=S(0,1)\cup S(0,r)$, where $S(0,R):=\{x:\ |x|=R\}$ and
$r<1$. Consider $\sigma\in\mathcal E^+(\Sigma)$ such that
$\sigma_{S(0,1)}$ is the rotationally symmetric probability measure,
while $\sigma_{S(0,r)}$ is an arbitrary nonzero measure. Then
$\lambda^\sigma_\Sigma=\sigma_{S(0,1)}$, because $\sigma_{S(0,1)}$
minimizes $\|\nu\|^2$ among all probability measures supported by
the closed unit ball (see~\cite{L}). Since the potential
of~$\sigma_{S(0,1)}$ takes constant values~$c_1$ and~$c_r$
on~$S(0,1)$ and~$S(0,r)$, respectively, and $c_r>c_1$
(see~\cite{L}), we get $L=c_r>c_1=\ell$.\end{example}

A crucial assumption in Example~\ref{1} is that $\kappa$ does not
satisfy the maximum principle. As is seen from Example~\ref{2}, this
restriction is not necessary in case $f\ne0$.

\begin{example}\label{2}Let $\mathrm X=\mathbb R^n$, $n\geqslant3$, $g=1$,
$\kappa(x,y)=|x-y|^{\alpha-n}$, where $\alpha\in(0,2]$,
$f(x)=|x-a|^{\alpha-n}$, where $a\in S(0,1)$ is fixed, and let
$\lambda_*$ minimize $G_f(\nu)$ among the probability measures
supported by~$S(0,1)$.  Then there are a constant~$q$ and a closed
neighborhood~$U$ of~$a$ on~$S(0,1)$ such that $W^f_{\lambda_*}(x)=q$
n.\,e.~in~$S_{\lambda^*}$ and $W^f_{\lambda_*}\bigl|_U>2q$
(see~\cite{DS2}). We define $\sigma$ to be~$\lambda_*$
on~$S_{\lambda^*}$, any nonzero $\nu\in\mathcal E^+$ on~$U$, and~$0$
elsewhere, and let $\Sigma:=S_{\lambda^*}\cup U$. Then
$\lambda^\sigma_\Sigma=\lambda_*$ and, consequently,
$\ell=q<2q\leqslant L$.\end{example}

\end{document}